\documentclass[11pt, a4paper]{amsart}
\usepackage[T1]{fontenc}
\usepackage[utf8]{inputenc}
\usepackage{geometry}                
\usepackage{graphicx}
\usepackage{amsmath, amsthm, amssymb, amscd, MnSymbol}
\usepackage{mathrsfs}
\usepackage{hyperref}
\usepackage[all]{xy}
\usepackage{tikz}

\theoremstyle{plain}
\newtheorem{thm}{Theorem}[section]
\newtheorem{lem}[thm]{Lemma}
\newtheorem{prop}[thm]{Proposition}
\newtheorem{cor}[thm]{Corollary}

\theoremstyle{definition}
\newtheorem{defn}{Definition}[section]
\newtheorem{conj}{Conjecture}[section]
\newtheorem{example}{Example}[section]

\theoremstyle{remark}
\newtheorem{rem}{Remark}[section]

\newcommand{\gl}{\mathrm{GL}}
\newcommand{\SL}{\mathrm{SL}}

\newcommand{\bbg}{\mathbb{G}}

\newcommand{\cf}{\mathcal{F}}
\newcommand{\co}{\mathcal{O}}

\newcommand{\cl}{\mathcal{L}}

\newcommand{\cp}{\mathcal{P}}

\newcommand{\xx}{\mathscr{X}}

\newcommand{\ka}{\mathfrak{a}}
\newcommand{\kg}{\mathfrak{g}}

\newcommand{\kp}{\mathfrak{p}}
\newcommand{\kt}{\mathfrak{t}}

\newcommand{\ko}{\mathfrak{o}}

\newcommand{\bn}{\mathbf{N}}
\newcommand{\bz}{\mathbf{Z}}
\newcommand{\br}{\mathbf{R}}

\newcommand{\bq}{\mathbf{Q}}

\newcommand{\bc}{\mathbf{C}}

\newcommand{\ba}{\mathbf{A}}

\newcommand{\baa}{\mathbf{a}}

\newcommand{\bnn}{\mathbf{n}}
\newcommand{\bii}{\mathbf{i}}
\newcommand{\bxx}{\mathbf{x}}
\newcommand{\byy}{\mathbf{y}}
\newcommand{\bjj}{\mathbf{j}}

\newcommand{\fq}{\mathbf{F}_{q}}
\newcommand{\val}{\mathrm{val}}
\newcommand{\lie}{\mathrm{Lie}}
\newcommand{\Ad}{\mathrm{Ad}}
\newcommand{\ad}{\mathrm{ad}}

\newcommand{\sch}{\mathrm{Sch}}
\newcommand{\Hom}{\mathrm{Hom}}
\newcommand{\ec}{\mathrm{Ec}}

\newcommand{\reg}{\mathrm{reg}}

\newcommand{\wt}{\mathrm{wt}}

\newcommand{\wtt}{\widetilde{T}}
\newcommand{\ep}{\epsilon}

\author{Zongbin \textsc{Chen}}

\address{EPFL SB Mathgeom/Geom, MA B1 447\\
Lausanne, CH-1015, Switzerland} 
\email{zongbin.chen@epfl.ch} 
\title{Truncated affine grassmannians and truncated affine Springer fibers for $\gl_{3}$}

\begin{document}
\maketitle

\begin{abstract}

We state a conjecture (see \S3.4) on how to construct affine pavings for cohomologically pure projective algebraic varieties, which admit an action of torus such that the fixed points and $1$-dimensional orbits are finite. Experiments on the affine grassmannian for $\gl_{3}$ under the guideline of this conjecture, together with the work of Berenstein-Fomin-Zelevinsky and Kamnitzer, have led to the conjecture that the truncated affine grassmannians for $\gl_{r+1}$ admit affine pavings.

For the group $\gl_{3}$, we construct affine pavings for the truncated affine grassmannians, and we use it to study the affine Springer fibers. In particular, we find a family of truncated affine Springer fibers which are cohomologically pure in the sense of Deligne.

\end{abstract}

\section{Introduction}

Let $k=\fq$, we fix an algebraic closure $\bar{k}$ of $k$. Let $F=k((\ep))$ be the field of Laurent series with coefficients in $k$, $\co=k[[\ep]]$ the ring of integers of $F$, $\kp=\ep k[[\ep]]$ the maximal ideal of $\co$. Let $\val:F^{\times}\to \bz$ be the discrete valuation on $F^{\times}$ normalized by $\val(\ep)=1$.

Let $G=\gl_{r+1}$, let $T$ be the maximal torus of $G$ of the diagonal matrices. Their Lie algebras will be denoted by the corresponding Gothic letters. Let $K=G(\co)$ be the standard maximal compact open subgroup of $G(F)$. The affine grassmannian $\xx=G(F)/K$ is the ind-$k$-scheme such that 
$$
\xx(\mathbf{F}_{q^{n}})=G(\mathbf{F}_{q^{n}}((\ep)))/G(\mathbf{F}_{q^{n}}[[\ep]]),\quad \forall \, n\in \bn.
$$

Let $B_{0}$ be the Borel subgroup of $G$ of the upper triangular matrices, let $X_{*}^{+}(T)$ be the semigroup of dominant cocharaters of $T$ with respect to $B_{0}$. For any $\lambda\in X_{*}^{+}(T)$ which is at the same time a dominant character of $G^{\vee}=\gl_{d}$, the \emph{geometric Satake isomorphism} of Ginzburg \cite{ginz}, and Mirkovic-Vilonen \cite{mv} states that the intersection cohomology of the affine Schubert variety $\sch(\lambda)=\overline{K\ep^{\lambda}K/K}$ gives a geometric realization of the irreducible representation $L_{\lambda}$ of $G^{\vee}$ of highest weight $\lambda$. For any $\mu\in X_{*}(T),\,\mu\prec \lambda$, Mirkovic and Vilonen have constructed a family of closed algebraic sub varieties which give a basis of the weight $\mu$ eigenspace $IH^{*}(\sch(\lambda))_{\mu}$: They are closures of irreducible components of the intersection 
$$
\sch(\lambda) \cap U_{0}^{-}(F)\ep^{-\mu}K/K,
$$ 
where $U_{0}^{-}$ is the unipotent radical of the Borel subgroup $B_{0}^{-}$ of $G$ opposite to $B_{0}$. These algebraic sub varieties are called \emph{Mirkovic-Vilonen cycles}. It is known that they have the same dimension $\rho(\lambda+\mu)$, where $\rho$ is the half sum of the positive roots of $\kg$ with respect to $B_{0}$.

We will identify $X_{*}(T)\otimes \br$ with $\kt$. Let $\cp(T)$ be the set of Borel subgroups of $G$ containing $T$. For any $x\in \xx$, $B\in \cp(T)$, let $f_{B}(x)\in X_{*}(T)\subset \kt$ be the unique cocharacter $\nu$ such that $x\in U_{B}(F)\ep^{\nu}K/K$, where $U_{B}$ is the unipotent radical of $B$. Let $\ec(x)$ be the convex envelope of $(f_{B}(x))_{B\in \cp(T)}$. Given a generic point $x$ in a Mirkovic-Vilonen cycle, the polytope $\ec(x)$ is independent of the choice of $x$, it is called \emph{Mirkovic-Vilonen polytope}. It is known that the Mirkovic-Vilonen cycles and Mirkovic-Vilonen polytopes determine each other.

Our interest in the Mirkovic-Vilonen cycles comes from the problem of paving the affine Springer fibers. Recall that for a regular element $\gamma\in \kt(\co)$, the affine Springer fiber at $\gamma$,
$$
\xx_{\gamma}=\{g\in G(F)/K\,\mid\,\Ad(g^{-1})\gamma \in \kg(\co)\},
$$ 
has been introduced by Kazhdan and Lusztig \cite{kl}. It has been used by Goresky, Kottwitz and Macpherson \cite{gkm1} to prove the fundamental lemma of Langlands-Shelstad in the unramified case, under the following hypothesis:
 
\begin{conj}[Goresky-Kottwitz-Macpherson]\label{gkmconj}

The cohomology of $\xx_{\gamma}$ is pure in the sense of Deligne, i.e. the eigenvalues of the action of frobenius $\mathrm{Fr}_{q}$ on $H^{i}(\xx_{\gamma,\bar{k}}, \overline{\bq}_{l})$ have an absolute value of $q^{i/2}$ for any embedding $\overline{\bq}_{l}\to \bc$. 

\end{conj}

The conjecture has been verified in several particular cases. In \cite{gkm2}, \cite{l}, \cite{chen1}, the authors have found affine pavings of $\xx_{\gamma}$. 

The affine Springer fibers have a large symmetry group. The free abelian group $\Lambda$ generated by $\chi(\ep),\, \chi\in X_{*}(T)$ acts simply and transitively on the set of irreducible components of $\xx_{\gamma}$. In our paper \cite{chen2}, we construct a fundamental domain $F_{\gamma}$ of $\xx_{\gamma}$ with respect to this action, which is a projective algebraic variety of finite type. And we are able to reduce the above conjecture to:

\begin{conj}\label{fundamental}
The cohomology of $F_{\gamma}$ is pure in the sense of Deligne.
\end{conj}

Our method to prove this conjecture is again to construct affine pavings of $F_{\gamma}$. In \S3.4, we state two conjectures on how  to construct affine pavings for cohomologically pure projective algebraic varieties, which admit an action of torus such that the fixed points and $1$-dimensional orbits are finite. At the beginning of \S4, we explained how to adapt this method to $F_{\gamma}$.

To carry out the program, it is important to study the polytope $\ec(x)$ for any $x\in \xx$ and the truncated affine grassmannian
$$
\xx(\ec(x)):=\{y\in \xx\mid \ec(y)\subset \ec(x),\, \nu_{G}(y)=\nu_{G}(x)\},
$$
where for $x=gK$, we set $\nu_{G}(x)=\val(\det(g))$. The requirement $\nu_{G}(y)=\nu_{G}(x)$ means that $x$ and $y$ should lie on the same connected component of $\xx$.

More precisely, for each $B\in \cp(T)$, we need to understand the contracting cell
$$
C_{B}(\ec(x)):=\{y\in \xx(\ec(x))\mid f_{B}(y)=f_{B}(x)\}.
$$

Inspired by the results of Berenstein, Fomin and Zelevinsky \cite{bfz} on Lusztig parametrization of totally positive matrices in $U_{0}$, and the results of Kamnitzer \cite{k1} on the Mirkovic-Vilonen cycles, we are led to the following conjecture: 

\begin{conj}

For any $x\in \xx$, there exists an element $w$ in the Weyl group $W$ of $G$ such that $w\cdot \xx(\ec(x))$ is a Mirkovic-Vilonen cycle. Further more, the truncated affine grassmannian $\xx(\ec(x))$ admits an affine paving, and so is cohomologically pure. 
\end{conj}

In this paper, we verify this conjecture for the group $\gl_{3}$. We give two affine pavings, one using the non-standard paving in \cite{chen1}, the other following the general method in \S3.4.

To carry out the second approach, we determine precisely the contracting cells $C_{B}(\ec(x))$. In particular, we prove that they are all isomorphic to affine spaces. These results are then used to study the affine Springer fibers for the group $\gl_{3}$. In particular, we find a family of cohomologically pure ``\emph{truncated}'' affine Springer fibers. Let $\{E_{i}, F_{i}, H_{i}\}_{i=1}^{2}$ be the generator of $U_{q}(\mathfrak{gl}_{3})$ lifting $\{e_{i}, f_{i}, h_{i}\}_{i=1}^{2}$  of $\mathfrak{gl}_{3}$. Recall that there is a crystal structure on the set of all the MV-cycles (resp. polytopes), i.e. an action of $E_{i},\,F_{i}$ on them. Let $\bjj$ be a finite sequence of alternating $1, 2$ of length $l$. Let 
$$
E_{\bjj}:=E_{j_{1}}E_{j_{2}}\cdots E_{j_{l}}.
$$

For any regular element $\gamma\in \kt(\co)$, we can suppose without loss of generality that
$$
\val(\alpha_{12}(\gamma))=n_{1},\quad \val(\alpha_{23}(\gamma))=\val(\alpha_{13}(\gamma))=n_{2},
$$
and $n_{1}\geq n_{2}$. Let $\bii=(121)$ and $\bnn=(n_{1},n_{2}, n_{2})$, let $P^{\bii}(\bnn)$ be the Mirkovic-Vilonen polytope of Lusztig data $\bnn$ with respect to the reduced word $\bii$ of the longest element $w_{0}$.

\begin{thm}

The truncated affine Springer fibers
$$
\xx(E_{\bjj}\cdot P^{\bii}(\bnn))\cap \xx_{\gamma}
$$
admits an affine paving, so it is cohomologically pure.

\end{thm}

\subsection*{Notation}

Let $\Phi=\Phi(G,T)$ be the root system of $G$ with respect to $T$, let $(a_{ij})$ be its Cartan matrix. Let $W$ be the Weyl group of $G$ with respect to $T$. For any subgroup $H$ of $G$ which is stable under the conjugation of $T$, we write $\Phi(H,T)$ for the roots appearing in $\lie(H)$. For $B\in \cp(T)$, let $\Phi_{B}^{+}=\Phi_{B}^{+}(G,T)$ be the set of positive roots of $G$ with respect to $B$. For $B_{0}$, we simplify it to $\Phi^{+}$. Let $\Delta=\{\alpha_{1},\cdots, \alpha_{r}\}$ be the set of simple roots with respect to $B_{0}$, let $\{\varpi_{1},\cdots,\varpi_{r}\}$ be the corresponding fundamental weights. Let $\rho=\sum_{i=1}^{r}\varpi_{i}$ be the Weyl vector.

Let $s_{1},\cdots, s_{r}$ be the simple reflections associated to the simple roots in $\Delta$. Every element $w\in W$ can be written (non-uniquely) as a reduced expression $w=s_{i_{1}}\cdots s_{i_{k}}$, we say that $w$ is of length $\ell(w)=k$ and that $\bii=(i_{1},\cdots,i_{k})$ is a reduced word for $w$. Let $w_{0}$ be the longest element in $W$, let $m$ be its length, which is also the number of positive roots of $G$. For each simple root $\alpha_{i}$, let $\psi_{i}:\SL_{2}\to G$ be the $i$-th root subgroup of $G$. Let
$$
e_{i}=\psi_{i}\left(\begin{bmatrix}0&1\\ &0
\end{bmatrix} \right),\quad f_{i}=\psi_{i}\left(\begin{bmatrix}0&\\ 1 &0
\end{bmatrix}\right), \quad
h_{i}=\psi_{i}\left(\begin{bmatrix}1&\\ &-1
\end{bmatrix}\right).
$$
Let $\{E_{i}, F_{i}, H_{i}\}_{i=1}^{r}$ be the generator of the quantum group $U_{q}(\kg)$ lifting the $\{e_{i}, f_{i}, h_{i}\}_{i=1}^{r}$. For $w\in W$, let $\bar{w}$ be the lift of $w$ to $G$, defined using the lift of $\bar{s}_{i}:=\psi_{i}\Big( \begin{bmatrix} 0& 1\\ -1&0 \end{bmatrix}\Big)$.

To each $\alpha_{i}\in \Delta$, we have a unique maximal parabolic subgroup $P_{i}$ of $G$ containing $B_{0}$ such that $\Phi(N_{P_{i}},T)\cap \Delta=\alpha_{i}$, where $N_{P_{i}}$ is the unipotent radical of $P_{i}$. This gives a bijective correspondence between the simple roots in $\Delta$ and the maximal parabolic subgroups of $G$ containing $B_{0}$. Any maximal parabolic subgroup $P$ of $G$ is conjugate to certain $P_{i}$ by an element $w\in W$, the element $w\cdot\varpi_{i}$ doesn't depend on the choice of $w$, we denote it by $\varpi_{P}$. We also say that $\varpi_{P}$ is a chamber weight of level $i$. 


We use the $(G,M)$ notation of Arthur. Let $\cf(T)$ be the set of parabolic subgroups of $G$ containing $T$, let $\cl(T)$ be the set of Levi subgroups of $G$ containing $T$. For every $M\in \cl(T)$, we denote by $\cp(M)$ the set of parabolic subgroups of $G$ whose Levi factor is $M$. For $P\in \cp(M)$, we denote by $P^{-}$ the opposite of $P$ with respect to $M$. Let $X^*(M)=\Hom(M, \bbg_m)$ and $\ka_M^{*}=X^*(M)\otimes\br$. The restriction $X^{*}(M)\to X^{*}(T)$ induces an injection $\ka_{M}^{*}\hookrightarrow \ka_{T}^{*}$. Let $(\ka_{T}^{M})^{*}$ be the subspace of $\ka_{T}^{*}$ generated by $\Phi(M,T)$. We have the decomposition in direct sums
$$
\ka_{T}^{*}=(\ka_{T}^{M})^{*}\oplus \ka_{M}^{*}.
$$

The canonical pairing 
$$
X_{*}(T)\times X^{*}(T)\to \bz
$$ 
can be extended linearly to $\ka_{T}\times \ka_{T}^{*}\to \br$, with $\ka_{T}=X_{*}(T)\otimes \br$. For $M\in \cl(T)$, let $\ka_{T}^{M}\subset \ka_{T}$ be the subspace orthogonal to $\ka_{M}^{*}$, and $\ka_{M} \subset \ka_{T}$ be the subspace orthogonal to $(\ka_{T}^{M})^{*}$, then we have the decomposition
$$
\ka_{T}=\ka_{M}\oplus \ka_{T}^{M},
$$
let $\pi_{M},\,\pi^{M}$ be the projections to the two factors.

We identify $X_{*}(T)$ with $T(F)/T(\co)$ by sending $\chi$ to $\ep^{\chi}$. With this identification, the canonical surjection $T(F)\to T(F)/T(\co)$ can be viewed as
\begin{equation}\label{indexT}
T(F)\to X_{*}(T).
\end{equation}

We use $\Lambda_{G}$ to denote the quotient of $X_{*}(T)$ by the coroot lattice of $G$ (the subgroup of $X_{*}(T)$ generated by the coroots of $T$ in $G$). We have a canonical homomorphism
\begin{equation}\label{indexM}
G(F)\to \Lambda_{G},
\end{equation}
which is characterized by the following properties: it is trivial on the image of $G_{\mathrm{sc}}(F)$ in $G(F)$ ($G_{\mathrm{sc}}$ is the simply connected cover of the derived group of $G$), and its restriction to $T(F)$ coincides with the composition of (\ref{indexT}) with the projection of $X_{*}(T)$ to $\Lambda_{G}$. Since the morphism (\ref{indexM}) is trivial on $G(\co)$, it descends to a map
$$
\nu_{G}:\xx\to \Lambda_{G},
$$
whose fibers are the connected components of $\xx$.

An important remark: Our convention on the group action is different from that of Berenstein-Fomin-Zelevinsky \cite{bfz} and that of Kamnitzer \cite{k1}. For them, the group acts from the right, while for us it acts from the left. It suffices to take the inverse to pass between the two. Also, there is a difference on the chamber weights. For us, it points outward, while it points inward for them.

\section{Mirkovic-Vilonen cycles and polytopes}

\subsection{Lusztig parametrization of $U_{0}$}

Let $\kg=\kt\oplus\bigoplus_{\alpha\in \Phi}\kg_{\alpha}$ be the decomposition of $\kg$ into root spaces. For each root $\alpha\in \Phi$, let $U_{\alpha}$ be the unipotent subgroup of $G$ whose Lie algebra is $\kg_{\alpha}$. We have an isomorphism of group schemes $\bxx_{i}:\bbg_{a}\to U_{\alpha_{i}}$ sending $t\in \bbg_{a}$ to $1+te_{i}=\exp(te_{i})\in U_{\alpha_{i}}$, which gives a natural coordinate on $U_{\alpha_{i}}$. 

Let $\bii=(i_{1},\cdots,i_{m})$ be a reduced word of the longest element $w_{0}\in W$. Taking products, we get a morphism of algebraic varieties:
$$
U_{\alpha_{i_{m}}}\times\cdots\times U_{\alpha_{i_{1}}}\to U_{0},
$$
which is a birational isomorphism. In \cite{bfz}, Berenstein, Fomin and Zelevinsky determine an open dense sub variety of $U_{0}$ over which the morphism is biregular, and give an explicit inverse morphism to it.

Firstly, they construct a birational transformation on $U_{0}$. For a generic matrix $g$, let $[g]_{+}$ be the last factor $u$ in the Gaussian $LTU$-decomposition $g=vtu$, with $v\in U_{0}^{-},\,t\in T,\, u\in U_{0}$. The matrix entries of $[g]_{+}$ are rational functions of $g$, and explicit formulas can be written out in terms of minors of $g$. The map $\eta_{w_{0}}: U_{0}\to U_{0}$ sending $y $ to $x=[\bar{w}_{0}y^{t}]_{+}$ is a birational automorphism, where $y^{t}$ means the transposition of $y$. The inverse birational automorphism is given by $y=\eta_{w_{0}}^{-1}(x)=\bar{w}_{0}^{-1}[x\bar{w}_{0}^{-1}]^{t}_{+}\bar{w}_{0}$.

Define morphisms $\bxx_{\bii},\,\byy_{\bii}$ from $(\ba^{1}\backslash\{0\})^{m}$ to $U_{0}$ to be
\begin{eqnarray*}
\bxx_{\bii}(t_{1},\cdots,t_{m})&=&\bxx_{i_{m}}(t_{m})\cdot\cdots \cdot \bxx_{i_{1}}(t_{1}),\\
\byy_{\bii}(t_{1},\cdots,t_{m})&=&\eta_{w_{0}}^{-1}(\bxx_{\bii}(t_{1},\cdots,t_{m})).
\end{eqnarray*}

The reduced word $\bii$ of $w_{0}$ determines a sequence of Weyl group elements $w^{\bii}_{j}=s_{i_{1}}\cdots s_{i_{j}}$ and distinct negative coroots $\beta^{\bii}_{j}=-w^{\bii}_{j-1}\cdot \alpha_{i_{j}}^{\vee}$, and the chamber weights $\varpi^{\bii}_{j}=w^{\bii}_{j}\cdot \varpi_{i_{j}},\, j=1,\cdots, m$. Let $\Gamma^{\bii}=\{w^{\bii}_{j}\cdot \varpi_{l};\,j=1,\cdots, m;\, l=1,\cdots, r\}$, one can prove that $\Gamma^{\bii}$ consists of the fundamental weights and the $\varpi_{j}^{\bii}$, they are called $\bii$-chamber weights.

Finally, for a chamber weight $\varpi=w\cdot \varpi_{i}$ of level $i$, for any matrix $g\in G$, let $\Delta_{\varpi}(g)$ be the minor of $g$ of the first $i$ rows and the column set $w\cdot \{1,\cdots, i\}$.

\begin{prop}[Berenstein-Fomin-Zelevinsky]\label{para1}

The morphism $\byy_{\bii}: (\ba^{1}\backslash\{0\})^{m} \to U_{0}$ is a biregular isomorphism onto $\{y\in U_{0}\mid \Delta_{\varpi}(y^{-1})\neq 0,\,\forall \,\varpi\in \Gamma^{\bii}\}.$ 

\end{prop}


Further more, Berenstein, Fomin and Zelevinsky also determine the transition law $\byy_{\bii}\byy_{\bii'}^{-1}$ for any two reduced words $\bii,\,\bii'$ of $w_{0}$. Recall that two reduced words $\bii,\,\bii'$ are said to be related by a $d$-braid move involving $i,j$, starting at the position $k$, if
\begin{eqnarray*}
\bii&=& (\cdots,i_{k}, i,j, i,\cdots, i_{k+d+1},\cdots),\\
\bii'&=& (\cdots,i_{k}, j,i, j,\cdots, i_{k+d+1},\cdots).
\end{eqnarray*}
where $d$ is the order of $s_{i}s_{j}$.

For the linear groups of type $A$, we know that any two reduced words of $w_{0}$ can be related by a sequence of $2$ or $3$-braid moves.

\begin{prop}[Berenstein-Fomin-Zelevinsky]\label{para2}
Let $\bii,\,\bii'$ be two reduced words related by a $d$-braid move as above, $d=2,3$, starting at the position $k$. Suppose that $\byy_{\bii}(t_{\bullet})=\byy_{\bii'}(t'_{\bullet})$. Then

\begin{enumerate}
\item

We have $t_{j}=t_{j}'$ for $j\notin \{k+1,\cdots,k+d\}$,

\item

If $a_{i,j}=0$, then $d=2$ and $t_{k+1}'=t_{k+2},\, t_{k+2}'=t_{k+1}$.

\item

If $a_{i,j}=-1$, then $d=3$ and
$$
t_{k+1}'=\frac{t_{k+2}t_{k+3}}{t_{k+1}+t_{k+3}},\quad t_{k+2}'=t_{k+1}+t_{k+3},\quad t_{k+3}'=\frac{t_{k+1}t_{k+2}}{t_{k+1}+t_{k+3}}.
$$

\end{enumerate}

\end{prop}


\subsection{Mirkovic-Vilonen cycles}

For $M\in \cl(T)$, the natural inclusion of $M(F)$ in $G(F)$ induces a closed immersion of $\xx^{M}:=M(F)/M(\co)$ in $\xx^{G}$. For $P=MN\in \cf(T)$, we have the retraction
$$
f_{P}:\xx\to \xx^{M}
$$ 
which sends $gK=nmK$ to $mM(\co)$, where $g=nmk,\,n\in N(F),\, m\in M(F),\, k\in K$ is the Iwasawa decomposition. 

For $P\in \cf(T)$, we have the function $H_{P}:\xx\to \ka_{M}^{G}=\ka_{M}/\ka_{G}$ which is the composition 
$$
H_{P}:\xx\xrightarrow{f_{P}}\xx^{M}\xrightarrow{\nu_{M}}\Lambda_{M}\to \ka_{M}^{G}.
$$

\begin{prop}[Arthur]

Let $B',B''\in \cp(T)$ be two adjacent Borel subgroups, let $\alpha_{B',B''}^{\vee}$ be the coroot which is positive with respect to $B'$ and negative with respect to $B''$. Then for any $x\in \xx$, we have
$$
H_{B'}(x)-H_{B''}(x)=n(x,B',B'')\cdot \alpha_{B',B''}^{\vee},
$$
with $n(x, B', B'')\in \bz_{\geq 0}$.

\end{prop}

For any point $x\in \xx$, we denote by $\ec(x)$ the convex envelope in $\ka_{T}^{G}$ of the $H_{B}(x),\,B\in \cp(T)$. We are interested in the polytope $\ec(x)$ and the \emph{truncated affine grassmannian} $\xx(\ec(x))$, which is defined by
$$
\xx(\ec(x))=\{y\in \xx\mid \ec(y)\subset \ec(x), \,\nu_{G}(y)=\nu_{G}(x)\}.
$$
It turns out that they are intimately related to the Mirkovic-Vilonen cycles. More generally,

\begin{defn}
A family $(\lambda_{B})_{B\in \cp(T)}$ of elements in $\ka_{T}^{G}$ is called $(G,T)$-\emph{orthogonal} if it satisfies
$$
\lambda_{B'}-\lambda_{B''}= n_{B',B''}\cdot \alpha_{B',B''}^{\vee},\quad \text{ for some }n_{B',B''} \in \br,
$$
for any two adjacent Borel subgroups $B',\, B''\in \cp(T)$. It is called \emph{positive} $(G,T)$-orthogonal family if furthermore all the $n_{B',B''}$ are greater or equal to $0$. In this case, we write $\ec((\lambda_{B})_{B\in \cp(T)})$ for the convex envelope of the $\lambda_{B}$'s. 
\end{defn}

Given a positive $(G, T)$-orthogonal family $(\lambda_{B})_{B\in \cp(T)}$, let $\ec((\lambda_{B})_{B\in \cp(T)})$ be the convex envelope in $\ka_{T}^{G}$ of the $\lambda_{B},\,B\in \cp(T)$ and
$$
\xx((\lambda_{B})_{B\in \cp(T)})=\{y\in \xx\mid \ec(y)\subset \ec((\lambda_{B})_{B\in \cp(T)}), \,\nu_{G}(y)=\nu_{G}(\lambda_{B})\}.
$$

\begin{example}

A polytope $P$ is called \emph{Weyl polytope} if it is of the form $\ec((w\cdot \lambda)_{w\in W})$ for some $\lambda\in X_{*}^{+}(T)$. In this case, the truncated affine grassmannian
$
\xx(P)$ is nothing but the affine Schubert variety $\sch(\lambda)$.

\end{example}

\begin{defn}
Given $\mu_{1},\mu_{2}\in X_{*}(T),\,\mu_{2}-\mu_{1}\in X_{*}^{+}(T)$, the closure of the irreducible components of 
$$
U_{0}(F)\ep^{\mu_{2}}K/K\cap U_{0}^{-}(F)\ep^{\mu_{1}}K/K
$$
are called \emph{Mirkovic-Vilonen cycles} of coweight $(\mu_{1},\mu_{2})$.
\end{defn}

It is known that they have the same dimension $\rho(\mu_{2}-\mu_{1})$. Here and after, we will call them simply MV-cycles. Let $x$ be a generic point on the MV-cycle, the convex polytope $\ec(x)$ doesn't depend on the choice of $x$, and we will call it the \emph{MV-polytope} associated to the MV-cycle. It is known that they determine each other. It is easy to see that MV-cycles (resp. MV-polytopes) of coweight $(\mu_{1},\mu_{2})$ are isomorphic to that of coweight $(\lambda+\mu_{1},\lambda+\mu_{2}),\,\forall \lambda\in X_{*}(T)$. So we will not distinguish MV-cycles (polytopes) that are translation of each other by elements in $X_{*}(T)$. We will assume from now on that all the MV-cycles are of coweight $(-\mu, 0)$ for some $\mu\in X_{*}^{+}(T)$ unless stated otherwise.

In the work \cite{k1}, Kamnitzer determined all the possible MV-polytopes, and gives a description of  generic points on the MV-cycle, using the work of Berenstein, Fomin and Zelevinsky recalled in the previous section.

Let $(\lambda_{B})_{B\in \cp(T)}$ be a positive $(G,T)$-orthogonal family, we also write it as $(\lambda_{w})_{w\in W}$, with $\lambda_{w}=\lambda_{w\cdot B_{0}}$. Let $\bii$ be a reduced word for $w_{0}$, then $\lambda_{w^{\bii}_{1}},\cdots, \lambda_{w^{\bii}_{m}}$ determine a path along the edges of the polytope $P=\ec((\lambda_{B})_{B\in \cp(T)})$. Let $n_{1},\cdots, n_{m}$ be the lengths of the edges of this path. We call the $m$-tuple $(n_{1},\cdots,n_{m})$  the $\bii$-Lusztig datum of $P$. It is clear that the set of all the $\bii$-Lusztig datums, $\bii$ running through all the reduced words of $w_{0}$, determines uniquely the polytope $P$.

Let $\bii$ be a fixed reduced word of $w_{0}$, Kamnitzer \cite{k1} proves that the $\bii$-Lusztig datum determines uniquely the MV-polytope.

\begin{prop}[Kamnitzer]\label{lusztig}

Let $\bii',\,\bii''$ be two reduced words of $w_{0}$ related by a $d$-braid move, $d=2,3$, starting at the position $k$. Let $P$ be a MV-polytope, let $n'_{\bullet},\,n''_{\bullet}$ be the $\bii',\,\bii''$-Lusztig datum of $P$. Then

\begin{enumerate}
\item

We have $n'_{j}=n''_{j}$ for $j\notin \{k+1,\cdots,k+d\}$,

\item

If $\bii''$ is obtained from $\bii'$ by a $2$-move, then $n''_{k+1}=n'_{k+2},\, n''_{k+2}=n'_{k+1}$.

\item

If $\bii''$ is obtained from $\bii'$ by a $3$-move, then 
$$
n''_{k+1}=n'_{k+2}+n'_{k+3}-a,\quad n_{k+2}''=a,\quad n_{k+3}''=n'_{k+1}+n'_{k+2}-a,
$$
where $a=\min\{n'_{k+1}, n'_{k+3}\}$.
\end{enumerate}

\end{prop}

Given an $\bii$-Lusztig datum $n_{\bullet}$, let $P^{\bii}(n_{\bullet})=\ec((\lambda_{w})_{w\in W})$ be the MV-polytope with $\bii$-Lusztig datum $n_{\bullet}$, let $S^{\bii}(n_{\bullet})$ be the associated MV-cycle. Let $\mu=\sum_{j=1}^{m}n_{j}\beta_{j}^{\bii}$, then $S^{\bii}(n_{\bullet})$ is of coweight $(\mu,0)$.

Let
$$
A(n_{\bullet})=\{(t_{1},\cdots,t_{m})\in F^{m}\mid \val(t_{i})=n_{i}\text{ for all } i\}.
$$

\begin{prop}[Kamnitzer]\label{paramv}
Let $t_{\bullet}\in A(n_{\bullet})$, then 

\begin{enumerate}
\item
The point $[\byy_{\bii}(t_{\bullet})^{-1}]\in U_{0}(F)K/K\subset\xx$ lies in $ S^{\bii}(n_{\bullet})$, and
$$
f_{w^{\bii}_{j}\cdot B_{0}}([\byy_{\bii}(t_{\bullet})^{-1}])=\lambda_{w^{\bii}_{j}},\quad j=1,\cdots, m.
$$  

\item
The morphism $A(n_{\bullet})\to S^{\bii}(n_{\bullet})$ sending $t_{\bullet}$ to $[\byy_{\bii}(t_{\bullet})^{-1}]$ has dense open image.
\end{enumerate}
\end{prop}

In \cite{k2}, Kamnitzer determined the crystal structure on the MV-polytopes. Let $\bii$ be any reduced word of $w_{0}$ such that $i_{m}=i$. For any $(n_{\bullet})\in \bz_{\geq 0}^{m}$, we have

\begin{eqnarray*}
F_{i}\cdot P^{\bii}(n_{\bullet})&=&P^{\bii}(n_{1}, \cdots, n_{m-1}, n_{m}+1);\\
E_{i}\cdot P^{\bii}(n_{\bullet})&=& \begin{cases}
P^{\bii}(n_{1}, \cdots, n_{m-1}, n_{m}-1),& \text{ if } n_{m}\geq 1\\
0, &\text{ if } n_{m}=0.
\end{cases}
\end{eqnarray*}

One proves easily that the operation doesn't depend on the choice of $\bii$. The action of MV-cycle is induced by the action on the MV-polytope.




\section{Truncated affine grassmannians}

\subsection{Defining equations}

For $w\in W$, let $U_{w}=wU_{0}w^{-1}$. It is the unipotent radical of the Borel subgroup $w\cdot B_{0}\in \cp(T)$. Given $\lambda\in X_{*}(T)$, we have the semi-infinite orbit
$$
S_{\lambda}^{w}=U_{w}(F)\ep^{\lambda}K/K.
$$

\begin{lem}[Mirkovic-Vilonen]\label{semi-infinite}

We have
\begin{enumerate}
\item
$$
\overline{S_{\lambda}^{w}}=\bigsqcup_{\mu\prec_{w}\lambda}S_{\mu}^{w},$$ 
where $\prec_{w}$ is the Bruhat-Tits order on $X_{*}(T)$ with respect to $w\cdot B_{0}$.

\item

Inside $\overline{S_{\lambda}^{w}}$, the boundary of $S_{\lambda}^{w}$ is given by a hyperplane section under an embedding of $\xx$ in projective space.

\end{enumerate}

\end{lem}

Given a positive $(G,T)$-orthogonal family $(\lambda_{w})_{w\in W}$, we deduce from the above lemma that
$$
\xx((\lambda_{w})_{w\in W})=\bigcap_{w\in W} \overline{S_{\lambda_{w}}^{w}}.
$$
To find the defining equations of the truncated affine grassmannian is thus the same as to find those of the semi-infinite orbits.

For each chamber weight $\varpi=w\cdot \varpi_{i}$, we will define a function $D_{\varpi}:\xx\to \bz$ in the following way: Let $V_{\varpi_{i}}$ be the irreducible representation of $G$ of highest weight $\varpi_{i}$, let $v_{\varpi_{i}}$ be the highest weight vector, let $v_{\varpi}=\overline{w}\cdot v_{\varpi_{i}}$. The group $G_{F}$ acts on $V_{\varpi_{i},F}$ by extension of scalars. Given $v\in V_{\varpi_{i}, F}$, let $\val(v)$ be the unique $l\in \bz$ such that $v\in V_{\varpi_{i}}\otimes \kp^{l}$ while $v\notin V_{\varpi_{i}}\otimes \kp^{l+1}$. For $gK\in \xx$, we define
$$
D_{\varpi}(gK)=\val(g^{-1}\cdot v_{\varpi}).
$$

The function $D_{\varpi}$ has the property:
$$
D_{\varpi}(ugK)=D_{\varpi}(gK),\quad \forall\, u\in U_{w}(F),
$$
since $u\cdot v_{\varpi}=v_{\varpi}$. Combined with the Iwasawa decomposition, we get

\begin{lem}[Kamnitzer]

For each $w\in W$, the function $D_{w\cdot\varpi_{i}}$ takes the constant value $-\langle  \lambda_{w}, w\cdot \varpi_{i} \rangle$ on $S_{\lambda_{w}}^{w}$. In particular,
$$
S_{\lambda_{w}}^{w}=\{gK\in \xx\mid D_{w\cdot \varpi_{i}}(gK)=-\langle  \lambda_{w}, w\cdot \varpi_{i} \rangle,\,\forall\, i=1,\cdots, r\}.
$$

\end{lem}

Kamnitzer also related the function $D_{\varpi}$ with the minor $\Delta_{\varpi}$. In general, one has
$$
\val(\Delta_{\varpi}(g^{-1}))\geq D_{\varpi}(gK),
$$
and for each $L\in \xx$, there exists a representative $g\in G(F)$ such that $L=gK$ and one has equality in the above inequality. By lemma \ref{semi-infinite}, we have
$$
\overline{S_{\lambda_{w}}^{w}}=\{gK\in \xx\mid \val(\Delta_{w\cdot \varpi_{i}}(g^{-1}))\geq -\langle  \lambda_{w}, w\cdot \varpi_{i} \rangle,\,\forall\, i=1,\cdots, r\}.
$$

\begin{cor}

Given a positive $(G,T)$-orthogonal family $(\lambda_{w})_{w\in W}$, the truncated affine grassmannian $\xx((\lambda_{w})_{w\in W})$ is defined by the equations
\begin{equation}\label{grequation}
\val(\Delta_{w\cdot \varpi_{i}}(g^{-1}))\geq -\langle  \lambda_{w}, w\cdot \varpi_{i} \rangle,\quad \forall\, w\in W,\,\forall\, i=1,\cdots, r.
\end{equation}

\end{cor}

We call 
$$
S((\lambda_{w})_{w\in W}):=\bigcap_{w\in W} S^{w}_{\lambda_{w}}
$$
the \textit{Gelfand-Goresky-Macpherson-Serganova} (or GGMS) strata. It is defined by the equations
$$
D_{w\cdot \varpi_{i}}(g)= -\langle  \lambda_{w}, w\cdot \varpi_{i} \rangle,\quad \forall\, w\in W,\,\forall\, i=1,\cdots, r.
$$

\begin{cor}\label{closuremv}

Given a positive $(G,T)$-orthogonal family $(\lambda_{w})_{w\in W}$. If the GGMS strata $S((\lambda_{w})_{w\in W})$ is non-empty, then
$$
\xx((\lambda_{w})_{w\in W})=\overline{S((\lambda_{w})_{w\in W})}.
$$

\end{cor}

The affine grassmannian admits action of a torus. The torus $T$ acts on $\xx=G(F)/K$ by left multiplication. Let $\bbg_{m}$ be the rotation torus, it acts on $F=k((\ep))$ by 
$$
t*\ep^{n}=t^{n}\ep^{n},\quad\forall\, t\in \bbg_{m},\, n\in \bz.
$$ 
This action induces an action of $\bbg_{m}$ on $\xx=G(F)/G(\co)$. Let $\wtt=T\times \bbg_{m}$, it acts on $\xx$ with discrete fixed points and discrete $1$-dimensional orbits. To describe them, recall that $\wtt$ acts on $\kg_{F}$ as well, with the eigenspace decomposition
$$
\kg_{F}=\bigoplus_{m\in \bz}\ep^{m}\kt \oplus \bigoplus_{(\alpha,\,n)\in \Phi\times \bz} \ep^{n}\kg_{\alpha}+\ep^{N}\kg,\quad \forall\,N\gg 0.
$$ 
Let $U_{\alpha,n}$ be the subgroup of $G_{F}$ whose Lie algebra is $\ep^{n}\kg_{\alpha}$. The fixed points $\xx^{\wtt}$ is the same as $\xx^{T}$, the $1$-dimensional $\wtt$-orbits attached to $\ep^{\nu},\,\nu\in X_{*}(T)$, are
\begin{equation}\label{alphaorbit}
U_{\alpha,n}\ep^{\nu}K/K,\quad (\alpha,\,n)\in \Phi\times \bz.
\end{equation}
We will say that the orbit (\ref{alphaorbit}) is \textit{in the }$\alpha$ \textit{direction}.

The torus $\wtt$ acts on the truncated affine grassmannian $\xx((\lambda_{w})_{w\in W})$, with finitely many fixed points and finitely many $1$-dimensional $\wtt$-orbits. Let $\Gamma$ be its $1$-skeleton, i.e. it is the graph with vertices being the fixed points and with the edges being the $1$-dimensional $\wtt$-orbits. Two vertices are linked by an edge if and only if they lie on the closure of the corresponding $1$-dimensional $\wtt$-orbit. For each $w\in W$, let $\wt(\lambda_{w})$ be the number of edges in $\Gamma$ which is linked to $\lambda_{w}$. For each $\alpha\in \Phi_{w\cdot B_{0}}^{+}$, let $L_{w,\alpha}((\lambda_{w})_{w\in W})$ be the number of edges in $\Gamma$ which is linked to $\lambda_{w}$ in the $\alpha$ direction. We have
$$
\wt(\lambda_{w})=\sum_{\alpha\in \Phi_{w\cdot B_{0}}^{+}} L_{w,\alpha}((\lambda_{w})_{w\in W}).
$$
Let $T_{\lambda_{w}}(\xx((\lambda_{w})_{w\in W}))$ be the tangent space of $\xx((\lambda_{w})_{w\in W})$ at $\lambda_{w}$, let $T_{\lambda_{w}}(\xx((\lambda_{w})_{w\in W}))_{\alpha}$ be its $\alpha$-eigenspace with respect to the action of the torus $T$.

Writing out explicitly the equations in (\ref{grequation}), we get 

\begin{cor} \label{dimtangentmv}

We have 
$$
\dim(T_{\lambda_{w}}(\xx((\lambda_{w})_{w\in W})))=\wt(\lambda_{w}),
$$
and
$$
\dim(T_{\lambda_{w}}(\xx((\lambda_{w})_{w\in W}))_{\alpha})=L_{w,\alpha}((\lambda_{w})_{w\in W}).
$$

\end{cor}

\subsection{Truncated affine grassmannians and MV-cycles}

From now on, we work with the group $G=\gl_{3}$. Given $x\in\xx$, we are interested in the polytope $\ec(x)$ and the truncated affine Grassmannian
$
\xx(\ec(x)).
$

\begin{lem}\label{decompu}
Let $\bii$ be a reduced word of $w_{0}$. Any element in $U_{0}(F)K/K$ can be written as $[\byy_{\bii}(t_{\bullet})^{-1}]$ for some $(t_{\bullet})\in F^{3}$.
\end{lem}

\begin{proof}

For any $u\in U_{0}(F)$, we can always find $a\in U_{0}(\co)$ such that the minors 
$$
\Delta_{\varpi}(ua)\neq 0,\quad \forall \,\varpi\in \Gamma^{\bii}.
$$ 
By proposition \ref{para1}, we have $(t_{\bullet})\in F_{3}$ such that $ua=\byy_{\bii}(t_{\bullet})^{-1}$, so
$$
uK=uaK=\byy_{\bii}(t_{\bullet})^{-1}K.$$

\end{proof}

\begin{prop}\label{reduce to mv}

Let $x\in \xx$, let $\lambda_{w'}=H_{w'\cdot B_{0}}(x),\,\forall\,w'\in W$. Take $w\in W$ such that
$$
\langle w\cdot \rho,\, \lambda_{w}-\lambda_{ww_{0}} \rangle=\min_{w'\in W}\{\langle w'\cdot \rho,\, \lambda_{w'}-\lambda_{w'w_{0}} \rangle\}.
$$
Then $w^{-1}\cdot \ec(x)$ is a MV-polytope.

\end{prop}

\begin{proof}

Up to permutation by $W$ and translation by $X_{*}(T)$, we can suppose that $w=1$ and $\lambda_{1}=0$. In particular, we have $x\in U_{0}(F)K/K$.

For the group $\gl_{3}$, there are only $2$ reduced words for $w_{0}$:
$
\bii=(121), \, \bii'=(212).
$
Let $n_{\bullet}$ (resp. $n'_{\bullet}$) be the sequence of lengths of the edges along the path of $\ec(x)$ passing through the $\lambda_{w^{\bii}_{j}}$'s (resp. $\lambda_{w^{\bii'}_{j}}$'s), $j=1,2,3$. We need to prove that $n_{\bullet}, n'_{\bullet}$ satisfy the braid relation in proposition \ref{lusztig} (3).

By lemma \ref{decompu}, there exist $t_{\bullet}, t'_{\bullet}\in F^{3}$ such that
$$
x=[\byy_{\bii}(t_{\bullet})^{-1}]=[\byy_{\bii'}(t'_{\bullet})^{-1}].
$$

By proposition \ref{paramv} (1), we have
$$
n_{i}=\val(t_{i}) \text{ and } n'_{i}=\val(t_{i}'),\quad i=1,2,3.
$$ 
So $\ec(x)=P^{\bii}(n_{\bullet})\cap P^{\bii'}(n'_{\bullet})$.

By proposition \ref{para2}, we have the relation
$$
t_{1}'=\frac{t_{2}t_{3}}{t_{1}+t_{3}},\quad t_{2}'=t_{1}+t_{3},\quad t_{3}'=\frac{t_{1}t_{2}}{t_{1}+t_{3}}.
$$

There are two cases to consider. If $\val(t_{1}+t_{3})=\min\{\val(t_{1}),\val(t_{3})\}$, then $n_{\bullet},n'_{\bullet}$ satisfy the 3-braid relation in proposition \ref{lusztig} (3), and $\ec(x)$ is a MV-polytope.

If $\val(t_{1}+t_{3})>\min\{\val(t_{1}),\val(t_{3})\}$, then $\val(t_{1})=\val(t_{3})$. This implies that $n_{1}=n_{3}$ and $P^{\bii}(n_{\bullet})$ is a Weyl polytope. Let $P^{\bii}(n_{\bullet})=\ec((\mu_{w})_{w\in W})$. Since $\ec(x)\subsetneq P^{\bii}(n_{\bullet})$ and they share the same $\bii$-path along the edges, we get
\begin{eqnarray*}
\langle s_{2}\cdot \rho,\, \lambda_{s_{2}}-\lambda_{s_{2} w_{0}} \rangle &< &\langle s_{2}\cdot \rho,\, \mu_{s_{2}}-\mu_{s_{2} w_{0}} \rangle\\
&=& \langle  \rho,\, \mu_{1}-\mu_{w_{0}} \rangle \quad (\text{since } P^{\bii}(n_{\bullet})\text{ is a Weyl polytope})\\
&=& \langle  \rho,\, \lambda_{1}-\lambda_{w_{0}}\rangle. 
\end{eqnarray*}

This is in contradiction with the hypothesis that $\langle  \rho,\, \lambda_{1}-\lambda_{w_{0}}\rangle$ is minimal. So this case is impossible, and $\ec(x)$ is a MV polytope.

\end{proof}

In view of this result, we will call permutations by $W$ of the MV-polytopes (resp. cycles) the ``\emph{generalized}'' MV-polytopes (resp. cycles). The corollary \ref{closuremv} implies that $\xx(\ec(x))$ is irreducible, so we get

\begin{cor}

For any $x\in \xx$, the truncated affine grassmannian $\xx(\ec(x))$ is a generalized MV-cycle.

\end{cor}

\subsection{An affine paving}

In order to pave $\xx(\ec(x))$, we need the non-standard affine paving of $\xx$ found in \cite{chen1}: Let $I$ be the standard Iwahori subgroup of $G(F)$, i.e. it is the inverse image of $B_{0}$ under the natural reduction $G(\co)\to G(k)$. For $\baa\in \bz^{3}$, let $I_{\baa}=\Ad(\ep^{\baa})I$.

\begin{prop}\label{pave}

Let $\baa\in \bz^{3}$.

\begin{enumerate}

\item
For $\lambda\in X_{*}^{+}(T)$ such that $\lambda_{1}\geq \lambda_{2}=\lambda_{3}$, we have the affine paving
$$
\sch(\lambda)=\bigsqcup_{\lambda'\in \sch(\lambda)^{T}}\sch(\lambda)\cap I_{\baa}\ep^{\lambda'} K/K.
$$
The intersection $\sch(\lambda)\cap I_{\baa}\ep^{\lambda'} K/K=J_{\baa,\lambda,\lambda'}\ep^{\lambda'}K/K$, where $J_{\baa,\lambda,\lambda'}$ is the open and closed sub-$k$-variety of $I_{\baa}$ of the matrices $(x_{i,j})$ such that $x_{i,i}\in \co$ and that
$$
 \val(x_{i,j})\geq m_{i,j},\quad \forall i\neq j,
$$ 
where $m_{i,j}=\max(a_{i}-a_{j}+\frac{i-j}{3},\,\lambda_3-\lambda'_{j})$. Moreover, the inclusion
$$
\sch(\lambda)\cap I_{\baa}\ep^{\lambda'}K/K\subset\overline{(\sch(\lambda)\cap I_{\baa}\lambda''K/K)}
$$
implies that $\lambda'\prec_{I_{\baa}} \lambda''$, where $\prec_{I_{\baa}}$ means the Bruhat-Tits order with respect to $I_{\baa}$.

\item

For $\lambda\in X_{*}^{+}(T)$ such that $\lambda_{1}= \lambda_{2}\geq \lambda_{3}$, we have the affine paving
$$
\sch(\lambda)=\bigsqcup_{\lambda'\in \sch(\lambda)^{T}}\sch(\lambda)\cap I_{\baa}\ep^{\lambda'} K/K.
$$
The intersection $\sch(\lambda)\cap I_{\baa}\ep^{\lambda'} K/K=\hat{J}_{\baa,\lambda,\lambda'}\ep^{\lambda'}K/K$, where $\hat{J}_{\baa,\lambda,\lambda'}$ is the open and closed sub-$k$-variety of $I_{\baa}$ of the matrices $(x_{i,j})$ such that $x_{i,i}\in \co$ and that
$$
 \val(x_{i,j})\geq \hat{m}_{i,j},\quad \forall i\neq j,
$$ 
where $\hat{m}_{i,j}=\max(a_{i}-a_{j}+\frac{i-j}{3},\,-\lambda_1+\lambda'_{i})$. Moreover, the inclusion
$$
\sch(\lambda)\cap I_{\baa}\ep^{\lambda'}K/K\subset\overline{(\sch(\lambda)\cap I_{\baa}\lambda''K/K)}
$$
implies that $\lambda'\prec_{I_{\baa}} \lambda''$.

\end{enumerate}

\end{prop}

\begin{rem}

The proposition can also be used to calculate $\sch(\lambda)\cap U_{B}(F)\ep^{\lambda'} K/K$ for $B\in \cp(T)$, for this it is enough to take $\baa\in \bz^{3}$ to be positive enough with respect to $B$.

\end{rem}

Let $\iota:\xx\to \xx$ be the involution sending $gK$ to $(g^{t})^{-1}K$. We make the observation: Up to permutation by $W$ and involution by $\iota$, the MV-polytopes are of the form $P^{\bii}(n_{\bullet})$ with $\bii=(121)$ and $n_{1}\geq n_{3}\geq n_{2}\geq 0$. They are intersection of two triangles as shown in figure \ref{intersection}.

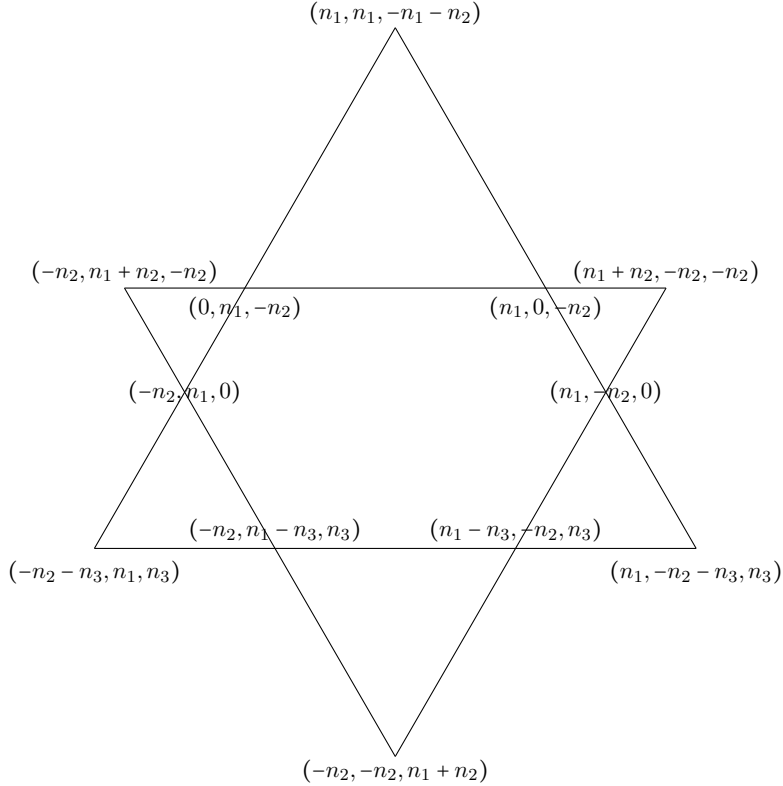
\begin{figure}[h]
\begin{center}
\begin{tikzpicture}[node distance = 2cm, auto, scale=0.8, transform shape]
\draw (-4.5, 1.73)--(4.5, 1.73);
\draw (-4.5, 1.73)--(0,-6.06);
\draw (0,-6.06)--(4.5, 1.73);
\draw (-5, -2.6)--(5, -2.6);
\draw (-5, -2.6)--(0, 6.06);
\draw (5, -2.6)--(0, 6.06);

\node at (0,6.3) {$(n_{1},n_{1},-n_{1}-n_{2})$};
\node at (-5, -3) {$(-n_{2}-n_{3},n_{1},n_{3})$};
\node at (5, -3) {$(n_{1},-n_{2}-n_{3},n_{3})$};

\node at (-4.5, 2) {$(-n_{2},n_{1}+n_{2},-n_{2})$};
\node at (4.5, 2) {$(n_{1}+n_{2},-n_{2},-n_{2})$};
\node at (0, -6.3){$(-n_{2},-n_{2}, n_{1}+n_{2})$};

\node at (-2.5, 1.4) {$(0, n_{1},-n_{2})$};
\node at (2.5,1.4) {$(n_{1},0, -n_{2})$};
\node at (-3.5,0) {$(-n_{2}, n_{1}, 0)$};
\node at (3.5, 0){$(n_{1},-n_{2}, 0)$};
\node at (-2, -2.3){$(-n_{2},n_{1}-n_{3}, n_{3})$};
\node at (2, -2.3){$(n_{1}-n_{3}, -n_{2}, n_{3})$};

\end{tikzpicture}
\caption{$P^{\bii}(n_{\bullet})$ as intersection of two triangles.}
\label{intersection}
\end{center}
\end{figure}

As a consequence of the Bruhat-Tits decomposition, we have 
$$
\xx(P^{\bii}(n_{\bullet}))=\sch(n_{1}+n_{2}, -n_{2},-n_{2})\cap \ep^{\baa}\cdot \sch(n_{3},n_{3},-n_{1}-n_{2}),
$$
where $\baa=(n_{1}-n_{3},n_{1}-n_{3},0)$.

\begin{prop}\label{pavemv}

For any point $x\in \xx$, the truncated affine grassmannian $\xx(\ec(x))$ admits an affine paving, so its cohomological groups are pure in the sense of Deligne.

\end{prop}

\begin{proof}

By proposition \ref{reduce to mv}, $\ec(x)$ is a MV-polytope up to permutation. Using the involution $\iota$ if necessary, we only need to work out the case when $\ec(x)=P^{\bii}(n_{\bullet})$ with $\bii=(121)$ and $n_{1}\geq n_{3}\geq n_{2}$. Let $\baa=(n_{1}-n_{3},n_{1}-n_{3},0)$. We have
\begin{eqnarray*}
\xx(P^{\bii}(n_{\bullet}))&=&\sch(n_{1}+n_{2}, -n_{2},-n_{2})\cap \ep^{\baa}\cdot \sch(n_{3},n_{3},-n_{1}-n_{2})\\
&=&\bigsqcup_{\lambda'\in X_{*}(T)\cap P^{\bii}(n_{\bullet})} \sch(n_{1}+n_{2}, -n_{2},-n_{2})\cap I_{\baa}\ep^{\lambda'}K/K.
\end{eqnarray*}

By proposition \ref{pave}, the intersections 
$$
\sch(n_{1}+n_{2}, -n_{2},-n_{2})\cap I_{\baa}\ep^{\lambda'}K/K
$$
are isomorphic to affine spaces, this gives us an affine paving of $\xx(P^{\bii}(n_{\bullet}))$.

\end{proof}


\subsection{A general method to construct affine pavings}

Let $V$ be a projective algebraic variety over $k$ admitting an action of torus $T$ such that the fixed points $V^{T}$ and the $1$-dimensional $T$-orbits $V^{T,1}$ are finite. Suppose that $V$ is cohomologically pure in the sense of Deligne. According to the localization theorem of Goresky, Kottwitz and Macpherson \cite{gkm4}, the $T$-equivariant cohomology of $V$ (hence the ordinary cohomology) can be calculated in terms of the fixed points and the $1$-dimensional $T$-orbits. It is natural to imagine that there is a general method to construct affine pavings from these datum.

Let $\Gamma$ be the graph with vertices $V^{T}$ and with edges associated to each $1$-dimensional $T$-orbits. Two vertices are linked by an edge if and only if they lie on the closure of the corresponding $1$-dimensional $T$-orbit. Let $\mathfrak{o}$ be a total order among the vertices of the graph $\Gamma$. We can associate to it an \textit{acyclic} oriented graph $(\Gamma, \ko)$ such that the source of each arrow is greater than its target with respect to $\ko$. For $v\in \Gamma$, denote by $n^{\ko}_{v}$ the number of arrows having source $v$.

\begin{defn}
The formal Betti number $b^{\ko}_{2i}$ associated to the order $\ko$ is defined as
$$
b_{2i}^{\ko}=\sharp\{v\in \Gamma:\,n^{\ko}_{v}=i\}.
$$
We call
$$
P^{\ko}(t)=\sum_{i}b_{2i}^{\ko}t^{2i}
$$
the formal Poincaré polynomial associated to the order $\ko$.
\end{defn}

\begin{defn}

For $P_{1}(t),\,P_{2}(t)\in \bz[t]$, we say that $P_{1}(t)< P_{2}(t)$ if the leading coefficient of $P_{2}(t)-P_{1}(t)$ is positive. 

\end{defn}

\begin{conj}\label{conjP}

Let $P_{V}(t)$ be the Poincaré polynomial of $V$, then
$$
P_{V}(t)=\min_{\ko}\{P^{\ko}(t)\},
$$
where $\ko$ runs through all the total orders among the vertices of $\Gamma$.
\end{conj}

There exists an inductive way to decompose $V$ into locally closed sub varieties according to the order $\ko$. To begin with, let $v_{0}$ be the vertex of maximal order. Take a co-character $\chi\in X_{*}(T)$ such that it is expanding on the tangent space $T_{v_{0}}(V)$, let
$$
V^{\ko}(v_{0}):=\{x\in V\mid \lim_{t\to 0}\chi(t)x=v_{0}\}.
$$

It can be proved that $V^{\ko}(v_{0})$ is open in $V$. Repeating this process inductively on the closed $T$-invariant sub variety $V\backslash V^{\ko}(v_{0})$, we get a decomposition of $V$ into locally closed sub varieties $V=\bigsqcup_{v\in V^{T}}V^{\ko}(v)$.

\begin{conj}\label{conjA}

Let $\ko$ be a total order such that $P^{\ko}(t)=P_{V}(t)$. Then the decomposition of $V$ according to the order $\ko$ is a \emph{generalized affine paving} of $V$. By the word ``generalized affine paving'', we mean a decomposition of $V$ into locally closed sub varieties which have the same compact support cohomology as a standard affine space $\ba^{n}$.

\end{conj}

\begin{rem}

These conjectures can be used to verify the cohomological purity of a projective algebraic variety $X$ which admits an action of torus such that the number of fixed points and that of $1$-dimensional orbits are finite. For this, it is enough to verify whether the above paving scheme gives a generalized affine paving of $X$.

\end{rem}

\subsection{More affine pavings}

We apply the above paving scheme to the truncated affine grassmannians, the result will be used to construct affine pavings for the affine Springer fibers in the next section.

First of all, we need to determine precisely the contracting cell 
$$
C_{B}(\ec(x))=\{y\in \xx(\ec(x))\mid f_{B}(y)=f_{B}(x)\},\quad \forall\, B\in \cp(T).
$$ 
As usual, we only need to work with the truncated affine grassmannian $\xx(P^{\bii}(n_{\bullet}))$, with $\bii=(121)$ and $n_{1}\geq n_{3}\geq n_{2}$. To simplify the notation, we enumerate the Weyl chambers clockwisely by $0,\cdots,5$, with the chamber $0$ corresponding to $B_{0}$. Let $B_{i}\in \cp(T)$ be the Borel subgroup corresponding to the $i$-th chamber, let $\lambda_{i}=H_{B_{i}}(x),\,i=0,\cdots, 5$.

\begin{thm}\label{cell}

For all the $B\in \cp(T)$, the contracting cell $C_{B}(\ec(x))=C_{B}(P^{\bii}(n_{\bullet}))$ is isomorphic to the standard affine space of dimension $n_{1}+2n_{2}+n_{3}$. More precisely, we have:

\begin{enumerate}

\item 

$$
C_{B_{0}}(P^{\bii}(n))=\begin{bmatrix} 1& \co &\kp^{n_{1}-n_{3}}\\
&1&\co\\  &&1
\end{bmatrix}^{-1} \begin{bmatrix}\ep^{n_{1}}&&\\ &1&\\ &&\ep^{-n_{2}}
\end{bmatrix}K/K;
$$

\item

$$
C_{B_{1}}(P^{\bii}(n))=\begin{bmatrix} 1& \co &\kp^{n_{1}-n_{3}}\\
&1&\\  & \co &1
\end{bmatrix} \begin{bmatrix}\ep^{n_{1}}&&\\ &\ep^{-n_{2}}&\\ &&1
\end{bmatrix}K/K;
$$

\item

$$
C_{B_{2}}(P^{\bii}(n))=\begin{bmatrix} 1& \co &\\
&1&\\ \kp^{n_{3}-n_{1}} & \co &1
\end{bmatrix} \begin{bmatrix}\ep^{n_{1}-n_{3}}&&\\ &\ep^{-n_{2}}&\\ &&\ep^{n_{3}}
\end{bmatrix}K/K;
$$

\item

$$
C_{B_{3}}(P^{\bii}(n))=\begin{bmatrix} 1& &\\ \co
&1&\\ \co& \kp^{n_{3}-n_{1}}  &1
\end{bmatrix} \begin{bmatrix}\ep^{-n_{2}}&&\\ &\ep^{n_{1}-n_{3}}&\\ &&\ep^{n_{3}}
\end{bmatrix}K/K;
$$

\item

$$
C_{B_{4}}(P^{\bii}(n))=\begin{bmatrix} 1& &\\ \co
&1& \kp^{n_{1}-n_{3}}\\ \co&  &1
\end{bmatrix} \begin{bmatrix}\ep^{-n_{2}}&&\\ &\ep^{n_{1}}&\\ &&1
\end{bmatrix}K/K;
$$

\item

$$
C_{B_{5}}(P^{\bii}(n))=\begin{bmatrix} 1& & \co \\ \co
&1& \kp^{n_{1}-n_{3}}\\ &  &1
\end{bmatrix}^{-1} \begin{bmatrix}\ep^{n_{1}}&&\\&1&\\ &&\ep^{-n_{2}}
\end{bmatrix}K/K.
$$

\end{enumerate}

\end{thm}

\begin{proof}

The calculations are similar, we only give the detail for the case (1). Since 
$$
\xx(P^{\bii}(n_{\bullet}))=\sch(n_{1}+n_{2}, -n_{2},-n_{2})\cap \ep^{\baa}\cdot \sch(n_{3},n_{3},-n_{1}-n_{2}),
$$
with $\baa=(n_{1}-n_{3},n_{1}-n_{3},0)$. To calculate 
$$
C_{B_{0}}(P^{\bii}(n_{\bullet}))=\xx(P^{\bii}(n_{\bullet}))\cap U_{0}(F)\ep^{\lambda_{0}}K/K,
$$
we only need to calculate
\begin{eqnarray*}
\sch(n_{1}+n_{2}, -n_{2},-n_{2})\cap U_{0}(F)\ep^{\lambda_{0}}K/K,
\end{eqnarray*}
and 
\begin{eqnarray*}
&&\Big[\ep^{\baa}\cdot \sch(n_{3},n_{3},-n_{1}-n_{2})\Big]\cap U_{0}(F)\ep^{\lambda_{0}}K/K\\
&&= \ep^{\baa}\cdot \Big[\sch(n_{3},n_{3},-n_{1}-n_{2})]\cap U_{0}(F)\ep^{\lambda_{0}-\baa}K/K \Big],
\end{eqnarray*}
and then take their intersection. By proposition \ref{pave}, we get
\begin{eqnarray*}
&&\sch(n_{1}+n_{2}, -n_{2},-n_{2})\cap U_{0}(F)\ep^{\lambda_{0}}K/K,\\
&&=\begin{bmatrix}1&\kp^{-n_{2}}&\co\\ &1&\co \\&&1
\end{bmatrix}\ep^{\lambda_{0}}K/K.
\end{eqnarray*}

Similarly, we get

\begin{eqnarray*}
&&\ep^{\baa}\cdot \Big[\sch(n_{3},n_{3},-n_{1}-n_{2})]\cap U_{0}(F)\ep^{\lambda_{0}-\baa}K/K \Big],\\
&&=\begin{bmatrix} 1&\co & \kp^{n_{1}-n_{3}}\\ &1& \kp^{-n_{3}}\\ &&1 \end{bmatrix}^{-1}\ep^{\lambda_{0}}K/K.
\end{eqnarray*}

Now taking intersection, we get the assertion (1).

\end{proof}

To pave $X=\xx(\ec(x))$, we begin with the following simple observation: 
For any vertex $v\in \Gamma$, let $\wt(v)$ be the number of edges in $\Gamma$ having end points at $v$, we call it the \emph{weight} of the vertex $v$. 
Theorem \ref{cell} tell us that $\wt(H_{B}(x))$ is equal to $\dim(X)$ for all $B\in \cp(T)$, and they are smaller than the weights of other vertices. 
Take $v_{0}=H_{B}(x)$, for any choice of $B\in\cp(T)$, we know that $X(v_{0})=C_{B}(\ec(x))$ is isomorphic to an affine space.

Let $X_{1}:=X\backslash X(v_{0})$ be the complement. 
We claim that $X_{1}$ are unions of several truncated affine grassmannians $X_{1}=\bigcup_{j} \xx(\ec(x^{(1)}_{j}))$. 
To see this, it is enough to remark that, for any $y\in X_{1}$, $\ec(y)$ is contained in one of the several biggest generalized MV-polytopes included in $\ec(x)$ while not including $v_{0}$. 
The algebraic variety $X_{1}$ is projective and $\wtt$-invariant, let $\Gamma_{1}$ be its $1$-skeleton. For any vertex $v\in \Gamma_{1}$, let $\wt_{1}(v)$ be the number of edges in $\Gamma_{1}$ having end points at $v$.

Take $j_{1}$ such that $\xx(\ec(x^{(1)}_{j_{1}}))$ is of maximal dimension. Let $v_{1}=H_{B_{1}}(x^{(1)}_{j_{1}})$, for some $B_{1}\in \cp(T)$, such that
$$
\wt_{1}(v_{1})=\min_{B\in \cp(T)}\{\wt_{1}(H_{B}(x^{(1)}_{j_{1}}))\}.
$$

It is easy to see that $v_{1}$ doesn't lie in the other $\xx(\ec(x^{(1)}_{j}))$'s, so 
$$
X_{1}(v_{1})=C_{B_{1}}(\ec(x_{j_{1}}^{(1)}))
$$ 
is isomorphic to an affine space.

Repeating the above process, we get a sequence of vertices $v_{i}\in \Gamma$ and $X_{i}$ in the following way: Let
$
X_{i}=X_{i-1}\backslash X_{i-1}(v_{i-1}).
$
As above, we have $X_{i}=\bigcup_{j} \xx(\ec(x^{(i)}_{j}))$. Let $\Gamma_{i}$ be the $1$-skeleton of $X_{i}$. For any vertex $v\in \Gamma_{i}$, let $\wt_{i}(v)$ be the number of edges in $\Gamma_{i}$ having end points at $v$.

Take $j_{i}$ such that $\xx(\ec(x^{(i)}_{j_{i}}))$ is of maximal dimension. Let $v_{i}=H_{B_{i}}(x^{(i)}_{j_{i}})$, for some $B_{i}\in \cp(T)$, such that
$$
\wt_{i}(v_{i})=\min_{B\in \cp(T)}\{\wt_{i}(H_{B}(x^{(i)}_{j_{i}}))\}.
$$

Again $v_{i}$ doesn't lie in the other $\xx(\ec(x^{(i)}_{j}))$'s, so 
$$
X_{i}(v_{i})=C_{B_{i}}(\ec(x_{j_{i}}^{(i)}))
$$ 
is isomorphic to an affine space.

In conclusion, we put the total order among the vertices to be $v_{0}\succ v_{1}\succ v_{2}\succ \cdots$, and we get the affine paving
$$
\xx(\ec(x))=\bigsqcup_{i} X_{i}(v_{i}). 
$$

Given any positive $(G,T)$-orthogonal family $(\lambda_{B})_{B\in \cp(T)}$, similar process applies to the truncated affine grassmannian $\xx((\lambda_{B})_{B\in \cp(T)})$.

\begin{thm}

Let $(\lambda_{B})_{B\in \cp(T)}$ be any positive $(G,T)$-orthogonal family, the truncated affine grassmannian $\xx((\lambda_{B})_{B\in \cp(T)})$ admits an affine paving. Further more, its Poincaré polynomial $P(t)$ satisfies
$$
P(t)=\min_{\ko}\{P^{\ko}(t)\},
$$
where $\ko$ runs through all the total order on the graph $\Gamma$.

\end{thm}

The result on the Poincaré polynomial is a consequence of the paving process.

\section{Applications to affine Springer fibers}

Given a regular element $\gamma\in \kt(\co)$, the affine Springer fiber at $\gamma$ is defined to be 
$$
\xx_{\gamma}=\{g\in G(F)/K\mid \Ad(g)^{-1}\gamma\in \kg(\co)\}.
$$
It is locally of finite type with dimension
$$
\dim(\xx_{\gamma})=\frac{1}{2}\val(\det(\ad(\gamma):\kg(F)/\kt(F))).
$$

The group $T(F)$ acts on $\xx_{\gamma}$ with one of the orbits $\xx_{\gamma}^{\reg}$ being dense open in $\xx_{\gamma}$, and the free abelian group $\Lambda$ generated by $\chi(\ep),\, \chi\in X_{*}(T)$ acts simply and transitively on the irreducible components of $\xx_{\gamma}$. A point $x=gK\in \xx_{\gamma}$ lies in $\xx_{\gamma}^{\reg}$ if and only if the image of $\Ad(g)^{-1}\gamma$ under the reduction map $\kg(\co)\to \kg$ is a regular element. According to \cite{gkm3}, they have the property:
\begin{equation}\label{gkmbound}
H_{B}(x)-H_{B'}(x)=\val(\alpha_{B,B'}(\gamma))\cdot \alpha_{B,B'}^{\vee}, \quad \text{ for any adjacent }B,B'\in \cp(T),
\end{equation}
where $\alpha_{B,B'}$ is the unique root which is positive with respect to $B$ while negative with respect to $B'$, and $\alpha_{B,B'}^{\vee}$ is the coroot associated to it.

Let $x_{0}\in \xx^{\reg}_{\gamma}$, let 
$$
F_{\gamma}=\{y\in \xx_{\gamma}\mid \ec(y)\subset \ec(x_{0}),\,\nu_{G}(y)=\nu_{G}(x_{0})\},
$$
it is called \emph{the fundamental domain} of $\xx_{\gamma}$ with respect to $\Lambda$. It is independent of the choice of $x_{0}$ up to translation by $\Lambda$. As mentioned in the introduction, the conjecture \ref{gkmconj} of Goresky, Kottwitz and Macpherson is equivalent to the conjecture \ref{fundamental}.

We believe that the general method in \S3.4 will give affine pavings of $F_{\gamma}$, although the conditions there are not quite satisfied. The torus $T$ acts on $F_{\gamma}$ with finitely many fixed points $F_{\gamma}^{T}$, but the $1$-dimensional $T$-orbits are not discrete. This can be overcome as follows: Let $F_{\gamma}^{T,1}$ be the union of $1$-dimensional $T$-orbits in $F_{\gamma}$, then $\wtt$ acts on it with finitely many $1$-dimensional orbits $F_{\gamma}^{\wtt,1}$. Let $\Gamma$ be the graph with vertices $F_{\gamma}^{\wtt}=F_{\gamma}^{T}$ and with edges associated to each $1$-dimensional $\wtt$-orbits. Two vertices are linked by an edge if and only if they lie on the closure of the corresponding $1$-dimensional $\wtt$-orbit. According to Goresky-Kottwitz-Macpherson \cite{gkm4} and Chaudouard-Laumon \cite{cl}, the $\wtt$-equivariant cohomology
$$
H^{*}_{\wtt}(F_{\gamma}):=H^{*}_{T}(F_{\gamma})\otimes_{H^{*}_{T}(\mathrm{pt})} H^{*}_{\wtt}(\mathrm{pt})
$$
of $F_{\gamma}$ can be expressed in terms of the graph $\Gamma$ as if the torus $\wtt$ acts on $F_{\gamma}$.

In the following, we work with the group $G=\gl_{3}$. We will construct a family of cohomologically pure ``\emph{truncated}'' affine Springer fibers, i.e. intersections of the form $\xx(\ec(x))\cap \xx_{\gamma}$ for some $x\in \xx$.

\subsection{A criterion}

We will determine whether the intersection $\xx_{\gamma}\cap C_{B}(\ec(x)), \, x\in \xx$, is isomorphic to an affine space.

For each $w\in W$, let $\lambda_{w}=f_{w\cdot B_{0}}(x)$. For $\alpha\in \Phi_{w\cdot B_{0}}^{+}$, let $l_{\alpha}^{w}=l_{\alpha}^{w\cdot B_{0}}$ be the number of $1$-dimensional $\wtt$-orbits contained in $\xx_{\gamma}\cap C_{w\cdot B_{0}}(\ec(x))$. As in corollary \ref{dimtangentmv}, we have
$$
\dim\big(T_{\lambda_{w}}(\xx_{\gamma}\cap \xx(\ec(x)))\big)=\sum_{\alpha\in \Phi_{w\cdot B_{0}}} l_{\alpha}^{w},
$$ 
and 
$$
\dim\big(T_{\lambda_{w}}(\xx_{\gamma}\cap \xx(\ec(x)))_{\alpha}\big)=l_{\alpha}^{w},
$$ 
where the subscript $\alpha$ refers to the $\alpha$-eigenspace of the tangent space under the action of the torus $T$. As usual, we only need to work out the case when $\ec(x)=P^{\bii}(n_{\bullet})$ with $\bii=(121), \, n_{1}\geq n_{3}\geq n_{2}$.

\begin{thm}\label{criterion}

Let $\bii=(121), \, n_{1}\geq n_{3}\geq n_{2}$, let $C_{B}=C_{B}(P^{\bii}(n_{\bullet}))$. Then

\begin{enumerate}

\item 

For $B=B_{0},\,s_{1}\cdot B_{0}$, the intersection $C_{B}\cap \xx_{\gamma}$ is isomorphic to an affine space if and only if 
$$
\sum_{\alpha\in \Phi_{B}^{+}} l_{\alpha}^{B}\leq n_{2}+n_{3}+\val(\alpha_{12}(\gamma)).
$$

\item

For $B= s_{2}s_{1}\cdot B_{0}, w_{0}\cdot B_{0}$, the intersection $C_{B}\cap \xx_{\gamma}$ is isomorphic to an affine space if and only if 
$$
\sum_{\alpha\in \Phi_{B}^{+}} l_{\alpha}^{B}\leq n_{2}+n_{3}+\val(\alpha_{12}(\gamma)).
$$

\item

For $B=s_{2}B_{0}$, the intersection $C_{B}\cap \xx_{\gamma}$ is isomorphic to an affine space if and only if 
$$
\sum_{\alpha\in \Phi_{B}^{+}} l_{\alpha}^{B}\leq n_{1}+n_{2}+\val(\alpha_{23}(\gamma)).
$$

\item

For $B=s_{1}s_{2}\cdot B_{0}$, the intersection $C_{B}\cap \xx_{\gamma}$ is isomorphic to an affine space if and only if 
$$
\sum_{\alpha\in \Phi_{B}^{+}} l_{\alpha}^{B}\leq n_{1}+n_{2}+\val(\alpha_{13}(\gamma)).
$$

\end{enumerate}

\end{thm}

\begin{proof}

The calculations are similar, we only give details for the first case. By theorem \ref{cell}, we have 
$$
C_{B_{0}}=\begin{bmatrix} 1& \co &\kp^{n_{1}-n_{3}}\\
&1&\co\\  &&1
\end{bmatrix}^{-1} \begin{bmatrix}\ep^{n_{1}}&&\\ &1&\\ &&\ep^{-n_{2}}
\end{bmatrix}K/K;
$$
 
Let $u=\begin{bmatrix}1&a&c\\ &1&b\\ &&1
\end{bmatrix}^{-1}
$, with $a\in \co, b\in \co, c\in \kp^{n_{1}-n_{3}}$. Then $u\ep^{\lambda_{0}}K\in \xx_{\gamma}$ if and only if $\Ad(u^{-1})\gamma \in \Ad(\lambda_{0})\kg(\co)$, i.e.
\begin{eqnarray*}
a(\gamma_{1}-\gamma_{2})\in \kp^{n_{1}},\quad b(\gamma_{2}-\gamma_{3})\in \kp^{n_{2}},\\
c(\gamma_{3}-\gamma_{1})+ab(\gamma_{1}-\gamma_{2})\in \kp^{n_{1}+n_{2}}.
\end{eqnarray*}

Let $c_{ij}=\val(\alpha_{ij}(\gamma))$. Solving the first two equations, we get 
$$
a\in \kp^{\max\{0, n_{1}-c_{12}\}}, \quad b\in \kp^{\max\{0, n_{2}-c_{23}\}}.
$$
The third equation tells us that $C_{B_{0}}\cap \xx_{\gamma}$ is isomorphic to an affine space if and only if
$$
\max\{0, n_{1}-c_{12}\}+\max\{0, n_{2}-c_{23}\}+c_{12}\geq \min\{n_{1}+n_{2},\, n_{1}-n_{3}+c_{13}\}.
$$

On the other hand, it is easy to calculate that
\begin{eqnarray*}
l_{\alpha_{12}}^{B_{0}}&=&n_{1}-\max\{0, n_{1}-c_{12}\},\\
l_{\alpha_{23}}^{B_{0}}&=&n_{2}-\max\{0, n_{2}-c_{23}\},\\ 
l_{\alpha_{13}}^{B_{0}}&=&\min\{c_{13}, n_{2}+n_{3}\}.
\end{eqnarray*}

Putting them in the above equation, we get the assertion (1).

\end{proof}

If the intersection $C_{w\cdot B_{0}}\cap \xx_{\gamma}$ is isomorphic to an affine space, its dimension will be equal to 
$$
\dim\big(T_{\lambda_{w}}(\xx_{\gamma}\cap \xx(\ec(x)))\big)=\sum_{\alpha\in \Phi_{w\cdot B_{0}}} l_{\alpha}^{w}.
$$

\subsection{A family of pure truncated affine Springer fibers}

In this section, we give a family of cohomologically pure truncated affine Springer fibers. Without loss of generality, we can suppose that the root valuation of $\gamma$ satisfies
$$
\val(\alpha_{12}(\gamma))=n_{1},\quad \val(\alpha_{23}(\gamma))=\val(\alpha_{13}(\gamma))=n_{2},
$$
such that $n_{1}\geq n_{2}$. Let $\bii=(121),\,\bnn=(n_{1},n_{2},n_{2})$. According to the equation (\ref{gkmbound}), we have
$$
F_{\gamma}=\xx(P^{\bii}(\bnn))\cap \xx_{\gamma}.
$$

Recall that there is a crystal structure on the set of all the MV-polytopes, i.e. an action of $E_{i},\,F_{i}$ on them. Let $\bjj$ be a finite sequence of alternating $1, 2$ of length $l$. Let 
$$
E_{\bjj}:=E_{j_{1}}E_{j_{2}}\cdots E_{j_{l}}.
$$

\begin{thm}\label{truncationpure}

The truncated affine Springer fibers
$$
[\xx(E_{\bjj}\cdot P^{\bii}(\bnn))]\cap \xx_{\gamma}
$$
admits an affine paving, so it is cohomologically pure.

\end{thm}

\begin{proof}

We will give a proof for $\bjj=(1)212\cdots 12$, the other case is similar. We will give an inductive affine paving by the length of $\bjj$.

To begin with, on proves easily that if the length of $\bjj$ is bigger than $2n_{2}$, then $[E_{\bjj}\cdot S^{\bii}(\bnn)]$ is empty and there is nothing to prove.

Let $\bjj=12\cdots 12$ of length $2n_{2}$, we claim that 
\begin{equation}\label{movetomv}
[\xx(E_{\bjj}\cdot P^{\bii}(\bnn))]\cap \xx_{\gamma}=\xx(E_{\bjj}\cdot P^{\bii}(\bnn)),
\end{equation}
from which the assertion follows by proposition \ref{pavemv}. Indeed, we have $E_{\bjj}\cdot P^{\bii}(\bnn)=P^{\bii}(n_{1}-n_{2}, n_{2},0)$, and the equation defining $\xx_{\gamma}$ is automatically satisfied by points in $\xx(P^{\bii}(n_{1}-n_{2}, n_{2},0))$.

Now we need to pave the difference between $[\xx(F_{1}E_{\bjj}\cdot P^{\bii}(\bnn))]\cap \xx_{\gamma}$ and $ [\xx(E_{\bjj}\cdot P^{\bii}(\bnn))]\cap \xx_{\gamma}$, or that between $[\xx(F_{2}E_{\bjj}\cdot P^{\bii}(\bnn))]\cap \xx_{\gamma}$ and $ [\xx(E_{\bjj}\cdot P^{\bii}(\bnn))]\cap \xx_{\gamma}$ according to $\bjj$. The paving goes exactly as the general paving scheme explained in \S3.4. We will indicate the order of paving by figures, and omit the detail of verification of conditions in theorem \ref{criterion}.

In the first case, i.e. when $\bjj=12\cdots 12$, the order to pave the complement is shown in figure \ref{order1}, where $2, 2', 3, 3', \cdots$ means that the order between these two sets of cells doesn't matter.

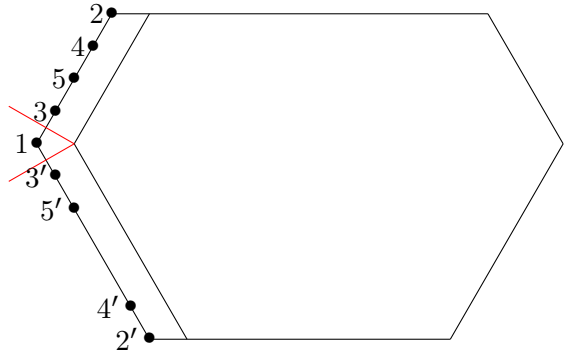
\begin{figure}[h]
\begin{center}
\begin{tikzpicture}[node distance = 2cm, auto]
\draw (-2.5, 1.73)--(2.5, 1.73);
\draw (-2.5, 1.73)--(-3.5, 0);
\draw (3.5, 0)--(2.5, 1.73);
\draw (-2, -2.6)--(2, -2.6);
\draw (-2, -2.6)--(-3.5,0);
\draw (2, -2.6)--(3.5, 0);
\draw (-2, 1.73)--(-3, 0);
\draw (-1.5, -2.6)--(-3,0);
\draw[red] (-3, 0)--(-3.87,0.5); 
\draw[red] (-3, 0)--(-3.87,-0.5);

\node at (-3.7,0) {1};
\node at (-3.5,0) {$\bullet$};

\node at (-2.5, 1.73){$\bullet$};
\node at (-2.7, 1.73){2};

\node at (-3.25, 0.43){$\bullet$};
\node at (-3.45, 0.43){3};

\node at (-2.75, 1.29){$\bullet$};
\node at (-2.95, 1.29){4};

\node at (-3.2, 0.86){$5$};
\node at (-3, 0.86){$\bullet$};

\node at (-2, -2.6){$\bullet$};
\node at (-2.3, -2.6){$2'$};

\node at (-2.25, -2.17){$\bullet$};
\node at (-2.55, -2.17){$4'$};


\node at (-3.25, -0.43){$\bullet$};
\node at (-3.5, -0.43){$3'$};

\node at (-3, -0.86){$\bullet$};
\node at (-3.3, -0.86){$5'$};

\end{tikzpicture}
\caption{Paving order when $\bjj=12\cdots 12 $}
\label{order1}
\end{center}
\end{figure}

In the second case, i.e. when $\bjj=212\cdots 12$, the order to pave the complement is shown in figure \ref{order2}:

\begin{figure}[ht]
\begin{center}
\begin{tikzpicture}[node distance = 2cm, auto]
\draw (-2.5, 1.73)--(2.5, 1.73);
\draw (-2.5, 1.73)--(-3.5, 0);
\draw (3.5, 0)--(2.5, 1.73);
\draw (-2, -2.6)--(2, -2.6);
\draw (-2, -2.6)--(-3.5,0);
\draw (2, -2.6)--(3.5, 0);
\draw (-2.25, -2.17)--(2.25, -2.17);

\node at (-2,-2.8) {1};
\node at (-2,-2.6) {$\bullet$};

\node at (2, -2.8){2};
\node at (2,-2.6) {$\bullet$};

\node at (-1.5,-2.8) {3};
\node at (-1.5,-2.6) {$\bullet$};

\node at (1.5, -2.8){4};
\node at (1.5, -2.6){$\bullet$};

\node at (-1,-2.8) {5};
\node at (-1,-2.6) {$\bullet$};

\node at (1,-2.8) {6};
\node at (1,-2.6) {$\bullet$};

\node at (0, -2.8){$\cdots$};

\end{tikzpicture}
\caption{Paving order when $\bjj=212\cdots 12 $}
\label{order2}
\end{center}
\end{figure}
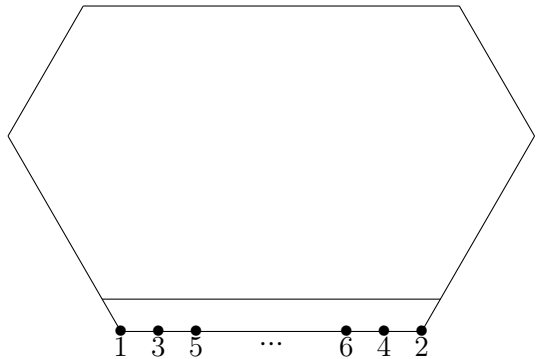

\end{proof}

\end{document}